\documentclass[11pt,oneside]{amsart}
\usepackage{amsthm,amssymb,amsfonts}
  \usepackage[margin=1.5in]{geometry}
 \usepackage{tikz}
  \usepackage{mathtools}

  \usetikzlibrary{positioning,intersections}
\usepackage{genyoungtabtikz}  
\usepackage{cleveref}
\usetikzlibrary{arrows,calc,through,backgrounds,matrix,decorations.pathmorphing}
\usetikzlibrary{decorations.pathreplacing}
\usepackage{scalefnt}
\crefname{defn}{Definition}{Definitions}
\crefname{thm}{Theorem}{Theorems}
\crefname{prop}{Proposition}{Propositions}
\crefname{lem}{Lemma}{Lemmas}
\crefname{cor}{Corollary}{Corollaries}
\crefname{conj}{Conjecture}{Conjectures}
\crefname{section}{Section}{Sections}
\crefname{subsection}{Subsection}{Subsections}
\crefname{eg}{Example}{Examples}
\crefname{figure}{Figure}{Figures}
\crefname{rem}{Remark}{Remarks}
\crefname{rmk}{Remark}{Remarks}
\crefname{equation}{equation}{equation}
\newtheorem*{Acknowledgements*}{Acknowledgements}

\Crefname{defn}{Definition}{Definitions}
\Crefname{thm}{Theorem}{Theorems}
\Crefname{prop}{Proposition}{Propositions}
\Crefname{lem}{Lemma}{Lemmas}
\Crefname{cor}{Corollary}{Corollaries}
\Crefname{conj}{Conjecture}{Conjectures}
\Crefname{section}{Section}{Sections}
\Crefname{subsection}{Subsection}{Subsections}
\Crefname{eg}{Example}{Examples}
\Crefname{figure}{Figure}{Figures}
\Crefname{rem}{Remark}{Remarks}
\Crefname{rmk}{Remark}{Remarks}

\newtheorem{thm}{Theorem}[section]

\newtheorem{prop}[thm]{Proposition}
\newtheorem{cor}[thm]{Corollary}
\theoremstyle{definition}
\newtheorem{defn}[thm]{Definition}

\newtheorem{eg}[thm]{Example}

\newtheorem{rmk}[thm]{Remark}

\newcommand{\Std}{\operatorname{Std}}
\newcommand{\SStd}{{\rm SStd}}

\usepackage{amsmath}
\mathchardef\mhyphen="2D

\allowdisplaybreaks

\newcommand{\ZZ}{{\mathbb Z}}
\newcommand{\Q}{{\mathbb Q}}

\newcommand{\GL}{\mathrm{GL}}

\newcommand{\Sym}{{\rm Sym}}

\newcommand{\SSTS}{\mathsf{S}}  
\newcommand{\SSTU}{\mathsf{U}}  
\newcommand{\SSTT}{\mathsf{T}}  
\newcommand{\SSTV}{\mathsf{V}}  
\newcommand{\sts}{\mathsf{s}}   
\newcommand{\stu}{\mathsf{u}}
\newcommand{\stv}{\mathsf{v}}
\newcommand{\stt}{\mathsf{t}}  
\newcommand{\Ten}{\mathbb{T}}  
\def\la{\lambda}

 \title[A generalised row removal formula for $p$-Kostka numbers]{Generalised row and column removal \\ phenomena   and $p$-Kostka numbers}

\author[C. Bowman]{Christopher Bowman}
\address{Department of Mathematics, City University London, Northampton Square, London EC1V 0HB, United Kingdom }
\email{Chris.Bowman.2@city.ac.uk}

\author[E. Giannelli]{Eugenio Giannelli}
\address{FB Mathematik, TU Kaiserslautern, Postfach 3049,
        67653 Kaisers\-lautern, Germany.}
\email{giannelli@mathematik.uni-kl.de}

\begin{document}

\begin{abstract}
We explain and generalise row and column removal phenomena for Schur algebras via isomorphisms between subquotients of these algebras.  In particular, we prove new reduction formulae for $p$-Kostka numbers and extension groups between Weyl modules and simple modules.  \end{abstract}

\maketitle

\section{Introduction}
 This paper is concerned with the study of the   representation  theory of the symmetric  and  general linear groups  over a field,  
  $\Bbbk$,  of  characteristic $p>0$.

  Given a partition $\lambda$ of $n$ into at most  $d$ non-zero parts, 
  we have    associated 
 $\GL_d$-modules:  
$ L(\lambda)$ 	 the simple module of highest weight $\lambda$;  
 $\Delta(\lambda) $	 (respectively $\nabla(\lambda) $)	 
  the  Weyl (respectively dual Weyl) module of highest weight $\lambda$; and 
$I(\lambda)$	 the injective cover of  $L(\lambda)$.  
Applying  the Schur functor to these modules, we obtain the 
simple modules 
$D(\lambda)$ (or zero); the Specht (and dual Specht) modules $S^\lambda$ (and  $S_\lambda$);
 and the Young modules $Y(\lambda)$ for the symmetric group $\mathfrak{S}_n$.  

One of the main open problems in the representation theory of general linear and symmetric groups is the following.

\smallskip

\noindent\textbf{Problem A}:
\textit{Given $\lambda$ and $\mu$   partitions of $n$, provide a combinatorial interpretation of  the decomposition numbers}
  $d_{\lambda \mu} = [\nabla(\lambda): L(\mu)].$ 
 
 \smallskip
 

  It is well-known that Problem  A     is equivalent to the following (see for instance \cite[Theorem 3.1]{JamesYoung} and \cite{ErdmannDecNumbers}).

 \smallskip

\noindent\textbf{Problem B}:
\textit{Given $\lambda$ and $\mu$   partitions of $n$, provide a combinatorial interpretation of  the multiplicities}
 $[\Sym^\lambda(\Bbbk^d): I(\mu)]=K_{\lambda \mu} = [{\rm ind} _{\mathfrak{S}_\lambda}^{\mathfrak{S}_n}(\Bbbk) :Y(\mu)]$.   
\smallskip

The multiplicities  $K_{\lambda \mu}$ 
are known as the $p$-Kostka numbers.  Young modules, and $p$-Kostka numbers in particular, 
have been extensively studied;  see for example 
\cite{ErdmannYoung1,ErdmannYoung2,ErdmannYoung3,FangHenkeKoenig,CGill,Grab, Henke,JamesYoung,Klyachko}.    
  In this article we prove a reduction formula for $p$-Kostka numbers. 
 Let $\lambda=(\lambda_1,\ldots,\lambda_d)$ and $\mu=(\mu_1,\ldots ,\mu_d)$ be partitions of $n$. 
For any fixed $1\leq r \leq d$, we  define partitions  
$$
\lambda^T = 
(\lambda_1,\lambda_2, \dots, \lambda_r),
\quad
\lambda^B = 
(\lambda_{r+1},\lambda_{r+2}, \dots, \lambda_{d}).
$$
We say that  $(\lambda, \mu)$ admits  a  horizontal  cut (after the $r$th row) if 
$|\lambda^T|=|\mu^T|$. 
Similarly, for $1\leq c \leq n$, we let
  $$\lambda^L= (\lambda'_1,\lambda_2',\ldots \lambda_c')' \quad \lambda^R= (\lambda'_{c+1},\ldots \lambda_{n}')'.$$  
 We say that $(\lambda, \mu)$ admits  a  vertical cut (after the $c$th column) if 
$|\lambda^L|=|\mu^L|$.

\begin{thm}\label{thm:main}
Let $(\lambda, \mu)$ be a pair of partitions of $n$
that 
 admits a horizontal row cut. Then 
$$K_{\lambda \mu}= K_{\lambda^T \mu^T}\cdot K_{\lambda^B \mu^B}.$$
Similarly, if $(\lambda, \mu)$ 
 admits a vertical cut, then $$K_{\lambda \mu}= K_{\lambda^L \mu^L}\cdot K_{\lambda^R \mu^R}.$$
\end{thm}

Similar reduction formulas were previously given for  (graded) decomposition numbers in \cite{Chuang,James,Donkin1}  and for
  the homomorphism spaces and extension groups  between Weyl and Specht modules  in  \cite{FL,LM,Donkin} and   \cite[4.2(17)]{Donkin}. 
Our approach  allows us to 
 give a simple proof and   extend  all of the aforementioned  (unquantised) results.  
 For example,  we also obtain new results concerning extension groups between a Weyl and a simple module, as follows.

 \begin{thm}
 If 
 $\lambda,\mu$ admit a horizontal cut, then
    \[
{ \rm Ext}^k_{S^{\Bbbk}_{n,d}}(\Delta(\lambda),L(\mu))=\bigoplus_{i+j=k}{ \rm Ext}^i_{S^{\Bbbk}_{m,r}}( \Delta( { \lambda^T}), L ( { \mu^T}))
\otimes
{ \rm Ext}^j_{S^{\Bbbk}_{n-m,d-r}}( \Delta( { \lambda^B}),L( { \mu^B})).  
 \]
 Similarly  if 
 $\lambda,\mu$ admit a vertical cut, then
   \[
{ \rm Ext}^k_{S^{\Bbbk}_{n,d}}(\Delta(\lambda),L(\mu))=\bigoplus_{i+j=k}{ \rm Ext}^i_{S^{\Bbbk}_{m,r}}( \Delta( { \lambda^L}), L( { \mu^L}))
\otimes
{ \rm Ext}^j_{S^{\Bbbk}_{n-m,d-r}}( \Delta( { \lambda^R}), L( { \mu^R})).  
 \]
Here $m$ is equal to the number of nodes  above the $r$th row (respectively to the left of the $c$th column) in the partition $\lambda$  or, equivalently, the partition $\mu$.  
 \end{thm}   
 
 The main idea of the proof is to construct explicit isomorphisms between  subquotients of the Schur algebra; 
 on the level of the cellular bases  these  simply break apart semistandard tableaux into `top' and `bottom' parts, in the obvious fashion.

The paper is structured as follows.  
 In the first two sections we give a review of the construction of the Schur algebra and 
 tensor space.    The exposition here does not follow the chronological development of the theory, but is cherry-picked  to be as  simple and combinatorial as possible.  We follow Doty--Giaquinto \cite{DG}  
  for the definition  of the Schur algebra via generators and relations.
 We also recall J. A. Green's  construction  of the co-determinant   basis of the Schur algebra
 and  Murphy's construction of an analogous basis of tensor space.  
 In Section \ref{Sec:proofs} we prove Theorem \ref{thm:main} by constructing 
explicit isomorphisms between subquotients of the  Schur algebra  and tensor space.
  In Section \ref{Sec:Sfunct} we recall    standard facts concerning the
   Schur functor and   
hence restate   the results of 
    \cref{Sec:proofs} in the  setting of the  symmetric group.  
 
\begin{Acknowledgements*}
The  authors  are grateful for the financial support received from the Royal Commission for the Exhibition of 1851 and from the ERC Advanced Grant 291512.
The first  author  also thanks the TU Kaiserslautern  for their hospitality during the early stages of this project.
\end{Acknowledgements*}

\section{The combinatorics of   tensor space}

We let $\Lambda_{n,d}$ denote the set of  {\sf compositions}  of $n$ into at most $d$ non-zero parts.  That is, the set of  sequences, $\lambda=(\lambda_1,\lambda_2,\dots,\lambda_d)$, of non-negative integers such that the sum $|\lambda|=\lambda_1+\lambda_2+\dots +\lambda_d$ equals $n$.  
     We let $\Lambda^+_{n,d}\subseteq \Lambda_{n,d}$ denote the subset 
     consisting of the sequences $\lambda=(\lambda_1,\lambda_2,\dots,\lambda_d)$ such that $\lambda_1 \geq \lambda_2 \geq \dots \geq \lambda_d$ and refer to such sequences as {\sf partitions}. 
%
%
%
%
%
With a partition, $\lambda$, is associated its {\sf Young diagram}, which is the set of nodes
\[[\lambda]=\left\{(i,j)\in\mathbb{Z}_{>0}^2\ \left|\ j\leq \lambda_i\right.\right\}.\]
  We let $\lambda'$ denote the {\sf conjugate partition} obtained by flipping the Young diagram $[\lambda]$ through the north-west to south-easterly diagonal.  
Given $\lambda,\mu \in \Lambda^+_{n,d}$ we say that $\lambda$ dominates $\mu$, 
and write $\lambda \trianglerighteq  \mu$ if 
$$
\sum_{1\leq i \leq r} \lambda_i
\geq
\sum_{1\leq i \leq r} \mu_i
$$
for all $1\leq r \leq d$.  
There is a surjective map $\Lambda_{n,d}\to\Lambda_{n,d}^+$ 
 given by rearranging the rows of a composition to obtain a partition  in the obvious fashion (for example if $n=9$ and $d=4$, then $(5,0,1,3) \mapsto (5,3,1,0)$).  
 Under the pullback of this map we obtain the dominance 
 ordering on the set of compositions, $\Lambda_{n,d}$, and we extend the notation in the obvious fashion.

Given $\lambda \in\Lambda^+_{n,d}$ and $\mu \in \Lambda_{n,d}$, we define a   $\lambda$-tableau of weight $\mu$ to be a map $\SSTT: [\lambda] \rightarrow \{1,\ldots, d\}$ such that $\mu_i=|\{ x \in [\lambda] : \SSTT(x)=i  \}|$ for $i\geq 1$.  
If $\SSTT$ is a $\lambda$-tableau of weight $\mu$, we say that $\SSTT$ is     {\sf semistandard} if the rows are weakly increasing from left to right and the columns are strictly increasing from top to bottom. 
 We let $\SSTT^\lambda$ denote the unique  element of $\SStd(\lambda,\lambda)$.  
 
 The set of all semistandard tableaux of shape $\lambda$ and weight $\mu$ is denoted    ${\SStd}(\lambda,\mu)$ and we  let $\SStd(\lambda,-):=\cup_{\mu\in \Lambda_{n,d}} \SStd(\lambda,\mu) $.  
For $d\geq n$, we have that   $\omega=(1^n,0^{d-n})$ belongs to $\Lambda^+_{n,d}$.  We refer to the tableaux  
 of weight $\omega$ 
  as the set of  {\sf standard tableaux};  we let   $\Std(\lambda):=\SStd(\lambda, \omega)$.  
We let $\stt^\lambda$ denote the element of $\Std(\lambda)$ in which the first row contains the entries $1,2, \ldots, \lambda_1$ the second row contains entries 
$\lambda_1+1,\lambda_1+2, \ldots, \lambda_2$ etc.  

\subsection{Symmetric groups and tensor space} 
 Fix a pair $n,d$ of positive integers
and let $ \Bbbk^d$ be the $\Bbbk$-module of rank $d$, spanned by the column vectors, $v_1, \ldots, v_d$,  over
$\Bbbk$    and let 
$
  \Ten =   (\Bbbk^d)^{\otimes n},
$ denote the 
the $n$th tensor power of $ \Bbbk^d$.  The module $\Ten $ is called {\sf tensor
  space}. 
Tensor space has a natural basis given by the {\sf elementary  tensors} of the form 
  $$
  v_{i_1} \otimes   v_{i_2}
 \dots   \otimes   v_{i_n},
   $$
   for some   $(i_1,i_2, \ldots , i_n) \in \{1,\ldots ,d\}^n$.   
We let $\mathfrak{S}_{\{1,2,\ldots, n\}}$ (or simply $\mathfrak{S}_n$) denote the  
{\sf symmetric group} of permutations of the set $\{1,2,\ldots, n\}$.  
The symmetric group   $\mathfrak{S}_n$  acts naturally on the right of $\Ten$. 
This action is given by      the place permutation of the subscripts of the elementary tensors,$$
(v_{i_1)} \otimes   v_{i_2}
 \dots   \otimes   v_{i_n}) \cdot  s
 =
 v_{s^{-1}(i_1)} \otimes   v_{s^{-1}(i_2)}
 \dots   \otimes   v_{s^{-1}(i_n)}.
$$
 and extending   $\Bbbk$-linearly. 

Given $\mu \in\Lambda_{n,d}$ and 
  $w$ an elementary tensor in $\Ten$, we say that the vector $w$ has {\sf weight} $\mu$ if 
$|\{i_x \mid 1\leq x \leq n , i_x=j \}|=\mu_j$, for all
$j \in  \{1,\ldots ,d\}$.  
We define the $\mu$-{\sf weight space} to be the subspace   $\Ten_\mu$ of $\Ten$ spanned by 
the set of elementary tensors of weight $\mu$.  


It is clear that the symmetric group acts by transitively permuting the set of elementary vectors of a given weight, $\mu \in \Lambda_{n,d}$. 
In particular, the elementary tensor 
$$
\underbrace{e_1  \otimes \dots\otimes  e_1}_{\mu_1}
\;\otimes\;
\underbrace{e_2\otimes \dots\otimes  e_2}_{\mu_2}
\;\otimes\;
\dots \;\otimes\;
\underbrace{e_d\otimes \dots\otimes  e_d}_{\mu_d}
$$
is a    generator of the $\mathfrak{S}_n$-module $\Ten_\mu$    and the stabiliser 
subgroup, denoted $\mathfrak{S}_\mu$, is equal to  the subgroup 
$$ \mathfrak{S}_{\{1,2,\dots,\mu_1\}}
\times
\mathfrak{S}_{\{\mu_1+1,\mu_1+2,\dots,\mu_2\}}
\times \dots\times 
\mathfrak{S}_{\{
n-\mu_{d}+1,n-\mu_{d}+2,\dots,n\}}. 
$$  
       
 \subsection{Murphy's basis of tensor space} 
We shall now define Murphy's basis of tensor space over several steps
\begin{itemize}
\item Let  $\lambda \in \Lambda^+_{n,d}$
 and  $\mu \in \Lambda_{n,d}$.  Given 
$\SSTS\in\SStd(\lambda,\mu)$ we define the {\sf row-reading element} 
  $e_{\SSTS}\in \Ten $ by recording  the entries of $\SSTS$, as read from left to right along  
   successive 
    rows, 
   as the subscripts in the tensor power.  For example, if
  $$
\Yvcentermath1\SSTS=\young(113,22)   $$ then
  $$e_{\SSTS}=v_1\otimes v_1\otimes v_3\otimes v_2\otimes v_2 .$$
  \item 
 For $\lambda \in \Lambda^+_{n, d}$, we have a corresponding Young subgroup $\mathfrak{S}_\lambda$ of $\mathfrak{S}_n$ given by the stabiliser of   $e_{\SSTT^\lambda}$.  We let 
 $\mathcal{O}_\lambda(e_{\SSTS})$ denote  the orbit sum of   vectors conjugate to $e_{\SSTS}$ under the natural right action of $\mathfrak{S}_\lambda$.     
 \item 
For $\stt \in \Std(\lambda)$ we let $d_\stt$ denote    the  permutation on $n$ letters  such that $(\stt^\lambda)d_\stt=\stt$.  
  \item
Given  $\SSTS \in \SStd(\lambda,\mu)$ and $\stt \in \Std(\lambda)$. We define
$$\rho_{\SSTS\stt}=(\mathcal{O}_{\lambda}(e_\SSTS)) d_\stt  $$
 \end{itemize}

\begin{thm}[Murphy \cite{Murphy}]\label{murphythm}
Tensor space $\Ten = (\Bbbk^d)^{\otimes n}$ is free as a $\mathbb{Z}$-module with basis given by
$$
\{\rho_{\SSTT\stt}\mid \SSTT \in \SStd(\lambda,\mu) , \stt\in \Std(\lambda), \lambda \in \Lambda^{+}_{n,d} , \mu \in \Lambda_{n,d}\}.
$$
\end{thm}
 
\begin{eg}Given $\lambda=(3,2)$, $\mu=(2,2,1)$ and $\SSTS$ as above, we have that
$$\rho_{\SSTS\stt^\lambda}=  v_1\otimes v_1\otimes v_3\otimes v_2\otimes v_2+v_1\otimes v_3\otimes v_1\otimes v_2\otimes v_2+v_3\otimes v_1\otimes v_1\otimes v_2\otimes v_2.$$
  \end{eg}

   \begin{eg}\label{example:1}
  Tensor space $\Ten = (\Bbbk^2)^{\otimes 4}$ is 16 dimensional.  
We have that  $\Lambda_{4,2}=\{(2,2),(3,1), (1,3), (4,0), (0,4)\}$. 
The semistandard tableaux, $\SSTS$, $\SSTT$, and  $\SSTU$ of weight $(2,2)$ are as follows 
$$
 \Yvcentermath1\young(11,22)
 \quad
 \Yvcentermath1\young(112,2)\quad
   \Yvcentermath1\young(1122)\;.	 
 $$
The standard tableaux  
  $\sts_1,\sts_2, \stt_1,\stt_2, \stt_3$ and $\stu$  are as follows
 $$
  \Yvcentermath1\young(12,34)\quad
    \Yvcentermath1\young(13,24)\quad
      \Yvcentermath1\young(123,4)\quad
            \Yvcentermath1\young(124,3)\quad
                  \Yvcentermath1\young(134,2)\quad
                        \Yvcentermath1\young(1234) \;.
 $$
The  space of vectors of  weight $(2,2)$  is 6-dimensional with basis  \begin{align*}
\rho_{\SSTS \sts_1} =& v_1\otimes v_1\otimes v_2\otimes v_2 \\
\rho_{\SSTS \sts_2}=& v_1\otimes v_2\otimes v_1\otimes v_2\\
  \rho_{\SSTT \stt_1}=& v_1\otimes v_1\otimes v_2\otimes v_2
+
v_1\otimes v_2\otimes v_1\otimes v_2
+
v_2\otimes v_1\otimes v_1\otimes v_2
\\
\rho_{\SSTT \stt_2}=& v_1\otimes v_1\otimes v_2\otimes v_2
+
v_1\otimes v_2\otimes v_2\otimes v_1
+
v_2\otimes v_1\otimes v_2\otimes v_1
\\
\rho_{\SSTT \stt_3}=& v_1\otimes v_2\otimes v_1\otimes v_2
+
v_1\otimes v_2\otimes v_2\otimes v_1
+
v_2\otimes v_2\otimes v_1\otimes v_1
\\
\rho_{\SSTU \stu}=  &v_1\otimes v_1\otimes v_2\otimes v_2
+
v_1\otimes v_2\otimes v_1\otimes v_2
+
v_1\otimes v_2\otimes v_2\otimes v_1
\\
&+ 
v_2\otimes v_1\otimes v_1\otimes v_2+
v_2\otimes v_1\otimes v_2\otimes v_1+
v_2\otimes v_2\otimes v_1\otimes v_1.
\end{align*}
\end{eg}
  
\section{The Schur algebra and the co-determinant  basis}
  Let $\Phi$ be the root system of type $A_{d-1}$: $\Phi =
\{\varepsilon_i - \varepsilon_j \mid 1 \le i \ne j \le d \}$.  Here
the $\varepsilon_i$s form the standard orthonormal basis of the euclidean
space $\mathbb{R}^d$.  Let $(\ ,\ )$ denote the inner product on this space
and define $\alpha_i = \varepsilon_i - \varepsilon_{i+1}$.  Then
$\{\alpha_1, \dots, \alpha_{d-1} \}$ is a base of simple roots and
$\Phi^+ = \{ \varepsilon_i - \varepsilon_j \mid i<j \}$ is the
corresponding set of positive roots.  We let $\alpha_i^\vee = 2\alpha_i/(\alpha_i,
\alpha_i)$ for $i=1, \dots, d$.

  The following definition of the Schur algebra over $\mathbb Q$ is due to Doty and Giaquinto \cite[Theorem 1.4]{DG} and is very much inspired by Lusztig's modified form of the quantum universal enveloping algebra.  
  
  \begin{defn}\label{gensandrelns}
{\setcounter{equation}{0}
\renewcommand{\theequation}{R\arabic{equation}}
The $\Q$-algebra $S^\Q_{n,d}$ is the associative
algebra (with $1$) given by generators $1_\lambda$ ($\lambda\in
\Lambda_{n,d}$), $e_{i,i+1}$, $f_{i,i+1}$ ($1\le i \le d-1$) subject to
the relations\begin{gather}
1_\lambda 1_\mu = \delta_{\lambda\mu} 1_\lambda, \quad
\sum_{\lambda\in \Lambda_{n,d}} 1_\lambda = 1 \\
e_{i,i+1} f_{j,j+1} - f_{j,j+1} e_{i,i+1} = \delta_{ij} \sum_{\lambda\in \Lambda_{n,d}}
(\alpha_i^\vee, \lambda)\, 1_\lambda \\
e_{i,i+1} 1_\lambda =
\begin{cases}
1_{\lambda+\alpha_i} e_{i,i+1}&
   \text{if $\lambda+\alpha_i \in \Lambda_{n,d}$}\\
0 & \text{otherwise}
\end{cases} \\
f_{i,i+1} 1_\lambda =
\begin{cases}
1_{\lambda-\alpha_i} f_{i,i+1} &
   \text{if $\lambda-\alpha_i \in \Lambda_{n,d}$}\\
0 & \text{otherwise}
\end{cases} \\
1_\lambda e_{i,i+1} =
\begin{cases}
e_{i,i+1} 1_{\lambda-\alpha_i} &
   \text{if $\lambda-\alpha_i \in \Lambda_{n,d}$}\\
0 & \text{otherwise}
\end{cases} \\
1_\lambda f_{i,i+1} =
\begin{cases}
f_{i,i+1} 1_{\lambda+\alpha_i} &
   \text{if $\lambda+\alpha_i \in \Lambda_{n,d}$}\\
0 & \text{otherwise}
\end{cases} 
%
\end{gather}
 }  
\end{defn}

\begin{rmk}
It was pointed out by Rouquier (see \cite[Introduction]{DGpre}) that the Serre relations $(R7)$ and $(R8)$ as stated in \cite[Theorem 1.4]{DG} follow from $(R1)$ to $(R6)$ and hence may be omitted. 
\end{rmk}

\begin{defn}
  For $1\leq i<j\leq d$, we inductively define elements 
 $$e_{i,j}=e_{i,j-1}e_{j-1,j} - e_{j-1,j}e_{i,j-1}
 \quad
 f_{i,j}=f_{i,j-1}f_{j-1,j} - f_{j-1,j}f_{i,j-1}.$$
We define the divided powers  
 \[
e_{i,j}^{[m]}=\frac{e_{i,j}^m}{m!}   \quad f_{i,j}^{[m]}=\frac{f_{i,j}^m}{m!}  \]
  The integral Schur
algebra $S^{\mathbb Z}_{n,d}$ is the subring of $S^{\mathbb Q}_{n,d}$ generated by all
divided powers.  \end{defn}

  \begin{prop}
We have an action of the Schur algebra $S^{\mathbb Q}_{n,d}$ on $\Ten$ defined as follows,
\begin{align*}
e_{i,i+1}(v_{j_1} \otimes \ldots \otimes v_{j_n})
&=
\sum_{\begin{subarray}c 1\leq a\leq n \\  j_a=i+1 \end{subarray}} (v_{j_1}\otimes \ldots  \otimes v_{j_a-1} \otimes  \ldots \  \otimes \ldots v_{j_n}) 
\\
f_{i,i+1}(v_{j_1} \otimes \ldots \otimes v_{j_n})
&=
\sum_{\begin{subarray}c 1\leq a\leq n \\  j_a=i \end{subarray}} (v_{j_1}\otimes \ldots  \otimes v_{j_a+1} \otimes  \ldots \  \otimes \ldots v_{j_n}) 
\\
1_\lambda  (v_{j_1}\otimes v_{j_2}\otimes \ldots v_{j_n})
&=
\begin{cases}
 (v_{j_1}\otimes v_{j_2}\otimes \ldots v_{j_n}) & \text{if the vector is of weight $\lambda$}\\
0 & \text{otherwise}
\end{cases}\end{align*}
\end{prop}
\begin{proof}
The relations (R1), (R3)--(R6) in Definition \ref{gensandrelns} are clear.  We now check that  (R2) holds.  
It is easy to see that 
$$
f_{i,i+1} e_{i,i+1}\left(v\right)=
\;\; \sum_{  \mathclap{\begin{subarray}c 1\leq a\leq n \\  j_a=i+1 \end{subarray}}}
\;\;
\left(v+\;\;\;\;\sum_{
\mathclap{\{b\neq a\mid  j_b=i\}}}
\;\;v_{j_1}\otimes \ldots  \otimes v_{j_a-1} \otimes  \ldots\otimes v_{j_b+1} \  \otimes \ldots v_{j_n}\right)   $$
$$e_{i,i+1} f_{i,i+1}\left(v\right)=
 \;\; \sum_{  \mathclap{ \begin{subarray}c 1\leq b\leq n \\  j_b=i \end{subarray}}}
\;\;
\left(v+  
\;\;\;\;\sum_{\mathclap{\{a\neq b\mid  j_a=i+1\}}}
\;\; v_{j_1}\otimes \ldots  \otimes v_{j_a-1} \otimes  \ldots\otimes v_{j_b+1} \  \otimes \ldots v_{j_n}\right)  
$$
for any $v=\left(v_{j_1} \otimes \ldots \otimes v_{j_n}\right)\in\Ten$.
It is now clear that $$
\left(e_{i,i+1} f_{i,i+1} - f_{i,i+1} e_{i,i+1}\right) v =  
 \left(|\{a\mid  j_a=i\}| -  |\{a\mid  j_a=i+1\}|\right)
 v,
$$as required.  
 \end{proof}

\begin{defn}\label{basiscombinatorics}  Given $1\leq i ,  j \leq d$ and   $\SSTT \in \SStd(\lambda,\mu)$,  we let $\SSTT(i,j)$ denote the number of entries equal to $j$ lying in the $i$th row  of $\SSTT$.
 Since $\SSTT$ is semistandard we have that $\SSTT(i,j)=0$ for $i>j$ and  $\sum_{1\leq i\leq d} \SSTT(i,j)=\mu_j$.  
\end{defn}

\begin{defn}\label{basiselements}Given $\SSTS, \SSTT \in \SStd(\lambda,\mu)$ we let 
 $$ 
 \xi_{\SSTS\lambda} = \prod^{1}_ { i = d } \left(\prod_{j=1}^{d} f_{i,j}^{[\SSTS(i,j)]}\right)  
 \quad
 \xi_{\lambda\SSTT} = \prod^{d}_ { i =1} \left(\prod_{j=1}^{d} e_{i,j}^{[\SSTT(i,j)]}\right)
$$
(notice the ordering on these products) and we define 
$$
\xi_{\SSTS\SSTT}= 		\xi_{\SSTS\lambda}	1_\lambda \xi_{\lambda\SSTT}
$$
\end{defn} 

\begin{eg}\label{2.6}Let  $\lambda=(3,3)$, $\mu=(2,2,1,1)$, and $\nu=(2,1,2,1)$.  We let $\SSTS $ and $\SSTT$ denote the tableaux 
\begin{equation*}\label{semi}
 \Yvcentermath1 \young(113,224)\quad  \quad   \young(112,334)\;,
\end{equation*}
respectively.
We have that $\SSTS\in \SStd(\lambda,\mu)$, $\SSTT\in \SStd(\lambda,\nu)$.   
We have that  $\SSTS(2,4)=1$,  $\SSTS(2,2)=2$,   $\SSTS(1,3)=1$,  $\SSTS(1,1)=2$,  and all other $\SSTS(i,j)$ are equal to zero.
Similarly,     $\SSTT(1,2)=1$,  $\SSTT(2,3)=2$,  $\SSTT(2,4)=1$ and all other $\SSTT(i,j)=0$.  
 Therefore, 
$$
 \xi_{\SSTS\SSTT}= 
 f^{[1]}_{1,3}f^{[1]}_{2,4} 
 1_\lambda 
 e^{[1]}_{2,4}
  e^{[2]}_{2,3} 
  e^{[1]}_{1,2}.
 $$
  \end{eg}
 
We now construct a basis of the Schur  algebra over $\mathbb Z$.  This basis is 
 known as the co-determinant basis and its original construction is due to J.~A.~Green  \cite{codet}; 
  it is generalised to the  ($q$-)Schur algebras of more general complex  reflection groups  by Dipper--James--Mathas \cite{DJM}.

 \begin{thm}\label{sdajfsadjksfdahjksadfhjkafsdhjklsadfhlj}

 The Schur algebra $S^{\mathbb Z}_{n,d}$ is free as a $\ZZ$-module with basis  $$
 \{ \xi_{\SSTS\SSTT}  \mid  \SSTS\in\SStd(\lambda,\mu), \SSTT\in \SStd(\lambda,\nu) \text{ for }\lambda \in \Lambda^+_{n,d} , \mu,\nu \in \Lambda_{n,d} \}.
 $$
If $\SSTS \in \SStd(\lambda,-)$, $\SSTT\in \SStd(\lambda,-)$  for some
      $\lambda\in\Lambda^+_{n,d}$, and $a\in S^{\mathbb Z}_{n,d} $ then 
    there exist scalars $r(a;\SSTS,\SSTU) \in \ZZ$, which do not depend on
    $\SSTT$, such that 
      \[ a\xi_{\SSTS\SSTT}   =\sum_{\SSTU \in
      \SStd(\lambda,-)}r(a;\SSTS,\SSTU)\xi_{\SSTU\SSTT}\quad
      {{\rm mod}\; (S^\ZZ_{n,d})^{\vartriangleright  \lambda}} \]
      where $(S^\ZZ_{n,d})^{\vartriangleright  \lambda}$ is the two-sided ideal generated by the idempotent $\sum_{\{\mu\in \Lambda_{n,d}\mid \mu\vartriangleright  \lambda\} }1_\mu$.  The ideal 
      $(S^\ZZ_{n,d})^{\vartriangleright  \lambda}$  is
        spanned by
      \[\{\xi_{\sf QR}\mid  
      {\sf Q,R}\in \SStd(\mu , - ), \mu\in \Lambda_{n,d}^+, \mu \vartriangleright  \lambda\}.\]
Moreover, the $\ZZ$-linear map $*:S^{\mathbb Z}_{n,d} \to S^{\mathbb Z}_{n,d} $ determined by
      $(\xi_{\SSTS\SSTT})^*=\xi_{\SSTT\SSTS}$, for all $\lambda\in\Lambda^+_{n,d}$ and
      all $\SSTS,\SSTT\in\SStd(\lambda,-)$, is an anti-isomorphism of $S^\ZZ_{n,d}$.
  Therefore the Schur algebra is a cellular algebra in the sense of \cite{GL}.

  \end{thm}

 \begin{proof}
  Having established the action of Doty and Guiaquinto's presentation on tensor space, and Murphy's basis of tensor space,  the above follows from 
   \cite[The semistandard basis theorem]{DJM}.  
 \end{proof}
\begin{defn}
Given  $\Bbbk$ an algebraically closed field of   characteristic $p\geq 0$, we define the   Schur
algebra $S^{\Bbbk}_{n,d}:=S^{\mathbb Z}_{n,d} \otimes \Bbbk$.    
\end{defn}

 \begin{defn} \label{definition: cell module}
Given  $\lambda\in\Lambda_{n,d}^+$, we define the  {\sf Weyl module} $\Delta^\ZZ(\lambda)$  to be the left $S^\ZZ_{n,d}$--module  with  basis 
  $$\{\xi_{\SSTS\SSTT^\lambda} + (S^\ZZ_{n,d})^{\vartriangleright \lambda} \mid \SSTS \in \SStd(\lambda,-) \}$$
  and the {\sf dual Weyl module} $\nabla^\ZZ(\lambda)$ to be the left $S^\ZZ_{n,d}$--module  with  basis 
    $$\{\rho_{\SSTS\stt^\lambda} + \Ten^{\vartriangleright \lambda} \mid \SSTS \in \SStd(\lambda,-) \},$$
    where $ \Ten^{\vartriangleright \lambda} $ is 
    left $S^\ZZ_{n,d}$--module of $\Ten$  with  basis 
  $$\{\rho_{\SSTS\stt}   \mid \SSTS \in \SStd(\mu,-),
   \stt \in \Std(\mu), \mu \vartriangleright \lambda \}$$
  We let  $\Delta^ \Bbbk(\lambda)$ (respectively
   $\nabla^ \Bbbk(\lambda)$)  denote the module  $\Delta^\ZZ(\lambda)\otimes_R \Bbbk$
   (respectively $\nabla^\ZZ(\lambda)\otimes_R \Bbbk$).  When the context is clear, we drop the ring over which the module is defined.  
 
  \end{defn}

   \begin{defn}
If a   module, $M$,  has a filtration of the form
$$0=M_0 \subset M_1\subset M_2 \subset \dots \subset M_k = M$$
where each $M_{i+1}/M_i$ for $1\leq i \leq k$ is isomorphic to some 
$\Delta(\lambda^{(i)})$  (respectively $\nabla(\lambda^{(i)})$)
for some $ \lambda^{(i)}\in \Lambda^+_{n,d}$, then we say that $M$ has a 
$\Delta $-  (respectively $\nabla $-) filtration and write 
$M\in \mathcal{F}(\Delta)$ (respectively $M\in \mathcal{F}(\nabla)$).   
  \end{defn}

%
%
%

Given any $\lambda\in\Lambda^+_{n,d}$  the Weyl module, $\Delta(\lambda) $, is equipped with a bilinear form $\langle\ ,\ \rangle_\lambda$  
 determined by
\[\xi_{\SSTU \SSTS}\xi_{\SSTT \SSTV}\equiv
  \langle \xi_{\SSTS\SSTT^\lambda},\xi_{\SSTT\SSTT^\lambda}\rangle_\lambda \xi_{\SSTU,\SSTV}
  \pmod   {(S^\ZZ_{n,d})^{\vartriangleright  \lambda}  }\] 
for   $\SSTS,\SSTT, \SSTU,\SSTV\in \SStd(\lambda, -  )$.
 We define    $L(\lambda)$ to be  the quotient of the corresponding Weyl module $\Delta(\lambda)$ by the radical of the bilinear form $\langle\ ,\ \rangle_\lambda$.   
Finally we denote by
$I(\lambda)$ the injective envelope of $L(\lambda)$ as an $S^\Bbbk_{n,d}$-module.


  \subsection{Generalised symmetric powers}
For $\lambda\in\Lambda^+_{n,d}$, $\mu\in\Lambda_{n,d}$ and $\stt \in \Std(\lambda)$ let $\mu( \stt) $ be the $\lambda$-tableau of weight $\mu$ obtained from $\stt$ by replacing each entry $i$ in $\stt$ by $r$ if $i$ appears in row $r$ of $\stt^\mu$. Given $\stt \in \Std (\lambda)$, we let $[\stt]_\mu$ denote the set   $\{  \sts \in \Std(\lambda)\mid \mu(\sts)=\mu(\stt)\}$.  
If $\SSTT\in\SStd(\lambda,\mu)$, we write $\stt \in \SSTT$ if $\mu(\stt) = \SSTT$. On the other hand, it will be convenient to say that $\mu(\stt)=0$, whenever $\mu(\stt)$ is not semistandard. Finally, for $\SSTS, \SSTT\in\SStd(\lambda,\mu)$ we set 
\begin{align}\label{youresumthatineedtogetby}
\rho_{\SSTS \SSTT }:=\sum_{\stt\in\SSTT} \rho_{\SSTS  \stt}. 
\end{align}
  \begin{rmk}
  In the case that $\mu = \omega$, the map $\omega:\Std(\lambda) \to \SStd(\lambda,\omega)$  is the bijective map which identifies 
standard tableaux with semistandard tableaux of weight $\omega$.  

\end{rmk}

 \begin{eg}
Let $n=4$ and $d=2$.  Adopting the same notation as in Example \ref{example:1} it is easy to observe that
there is a unique  element  of 
$\SStd(\lambda,(2,2))$   for each $\lambda \in \Lambda^+_{4,2}$.  These are the tableaux
 $\SSTS$, $\SSTT$ and $\SSTU$ of  \cref{example:1}.  
The pullback  under $\Std(\lambda) \to \SStd(\lambda,(2,2))$ is given by
  $$ [\sts_1]_{(2,2)}=\{\sts_1\}
 \quad
  [\stt_1]_{(2,2)}=\{\stt_1,\stt_2\} 
 \quad
 [\stu]_{(2,2)}=\{\stu\},
 $$for $\lambda$ equal to $(2,2), (3,1)$ and $(4)$, respectively.   
Therefore
\begin{align*}
\rho_{\SSTT \SSTT}=\rho_{\SSTT \stt_1}+\rho_{\SSTT \stt_2}
 =&
e_1\otimes e_2\otimes e_1\otimes e_2
+
e_2\otimes e_1\otimes e_1\otimes e_2 \\
&+
e_1\otimes e_2\otimes e_2\otimes e_1 
+
e_2\otimes e_1\otimes e_2\otimes e_1
+
2 e_1\otimes e_1\otimes e_2\otimes e_2.
\end{align*}
 \end{eg}

 \begin{defn}
Given $\mu \in \Lambda_{n,d}$, 
 we let 
 $$\Sym^\mu(\Bbbk^d)=\Sym^{\mu_1}(\Bbbk^d) \otimes \dots \otimes \Sym^{ d}(\Bbbk^d)$$    
 denote the  generalised symmetric   power of the natural $\GL_d$-module, $\Bbbk^d$.   
   \end{defn}

\begin{prop}\label{Crispysabsis}
The module $\Sym^\mu(\Bbbk^d)$ has  basis given by   sums of elements in the   Murphy basis of tensor space of   \cref{murphythm}, 
 as follows
$$ 
\{
\rho_{\SSTS \SSTT} \mid
 \SSTS \in \SStd(\lambda,\nu) , 
 \SSTT \in  \SStd(\lambda,\mu), {\lambda \in \Lambda^+_{n,d}}, {\nu \in \Lambda_{n,d}}   \}.  
$$
  \end{prop}

\begin{proof}
For each $\SSTS \in \SStd(\lambda,\nu) $ and $
 \SSTT \in  \SStd(\lambda,\mu)$ 
 we have that $\mathfrak{S}_\mu$ acts transitively on the set $\{\rho_{\SSTS   \stt  } \mid \stt \in \SSTT\}$. Moreover, the   
stabiliser of any element $\rho_{\SSTS   \stt  }$ is
 $\mathfrak{S}_\mu \cap d_\stt^{-1} \mathfrak{S}_\lambda d_\stt $ (see for example \cite[Proposition 4.4]{Mathasbook}).   Therefore the element $\rho_{\SSTS  \SSTT}$ is fixed by the action of $\mathfrak{S}_\mu$.
Hence, for every 
  $\SSTS \in \SStd(\lambda,\nu) $ and $
 \SSTT \in  \SStd(\lambda,\mu)$ 
 we have that $\rho_{\SSTS \SSTT } \in \Sym^\mu(\Bbbk^d)$.

The elements $\rho_{\SSTS  \stt }$ are linearly independent 
and the orbits  $\{\stt \mid \mu( \stt)=\SSTT\}$ 
for $\SSTT \in \SStd(\lambda,\mu)$  are disjoint. Therefore the elements  $\rho_{\SSTS  \SSTT }$ are linearly independent (over any field, as their coefficients in the sum in \cref{youresumthatineedtogetby} are all 0 or 1). 
 The result now follows from a dimension count using the formula 
 $$
[ \Sym^\mu(\Bbbk^d):\nabla(\lambda)] = |\SStd(\lambda,\mu)| $$
and the fact that   $\nabla(\lambda)$ has  basis indexed 
 by the set $\SStd(\lambda,\nu)$.  
\end{proof}

  \begin{prop}[Section 4.8 \cite{green}]\label{greenstuff}
  The injective indecomposable $S^{\Bbbk}_{n,d}$-modules are precisely the indecomposable summands of $\Sym^\mu(\Bbbk^d)$ for $\mu \in \Lambda^+_{n,d}$.  
 For  $\mu, \lambda \in \Lambda^+_{n,d}$, we have $$[\Sym^\mu(\Bbbk^d): I(\lambda)] = K_{\mu \lambda}= \dim L_\lambda(\mu)$$
 where the coefficients, $K_{\mu \lambda}$, are known as the $p$-Kostka numbers.  In particular,
$[\Sym^\mu(\Bbbk^d): I(\lambda)] =1$ for $\lambda=\mu$ and 0
unless $\mu \trianglelefteq \lambda$.  
  \end{prop}

\section{Isomorphisms between subquotients of  Schur algebras}\label{Sec:proofs}

In this section, we prove the main results of this paper.
In \cref{ACT1} we   consider 
certain subsets, $\Lambda^+_{n,d}(r,c,m)\subseteq \Lambda^+_{n,d}$.  
We recall the definition of generalised row cuts on pairs of partitions and  show that if 
$(\la,\mu)$  admit  such a cut and $\la \vartriangleright \mu$, then $\la$ and $\mu$ both belong to one our subsets $\Lambda^+_{n,d}(r,c,m)$.  
  In \cref{ACT2,ACT3,ACT4,ACT5} we construct explicit isomorphisms between certain subquotients  of the Schur algebras corresponding to the sets $\Lambda^+_{n,d}(r,c,m)$; all of these isomorphisms are given simply on the level of the tableaux bases.

The subquotients in which we are interested are of the following form.  
\begin{defn}
Let $P$ denote a partially ordered set and $Q$ denote a subset of $P$.  
We say that $Q$ is {\sf saturated} if for any $\alpha\in Q$ and $\beta \in P$ with $\beta \vartriangleleft \alpha$, we have that $\beta \in Q$. We say that $Q$ is {\sf co-saturated} if its complement in $P$ is saturated.  
If a set  is saturated, co-saturated, or the intersection of a saturated and a co-saturated set, we shall say that it is {\sf closed} under the dominance order.  
\end{defn}

\begin{defn}
Let $M$ be a $S^\Bbbk_{n,d}$-module, and $\pi \subseteq \Lambda^+_{n,d}$ denote some closed subset under the dominance order.  
We say that $M$ {\sf belongs} to $\pi$ if the simple composition factors of $M$ are labelled by weights from $\pi$.  
 We write $M \in \mathcal{F}_\pi(\Delta)$ (respectively  $M \in \mathcal{F}_\pi(\nabla)$) if $M$ has a $\Delta$-filtration (respectively $\nabla$-filtration)  in which the $\Delta$ (respectively $\nabla$) factors  are labelled by weights from $\pi$.  
\end{defn}

We shall use standard facts about saturated and co-saturated sets in what follows, referring to \cite[Appendix]{Donkin} for more details.  
 Much of the representation theoretic information is preserved under taking such subquotients. 
 In \cref{ACT5} we then deduce that higher extension groups and decomposition numbers are preserved under taking generalised row and column cuts, thus simplifying the proof and   extending the (unquantised) results of \cite{James,Donkin1,LM} and \cite[4.2(17)]{Donkin}.  
 In \cref{ACT6}, we consider the image of the generalised symmetric powers under these functors and hence prove that $p$-Kostka numbers are preserved under generalised row and column removal.

\subsection{Combinatorics of partitions and generalised  row cuts}\label{ACT1} We now recall the combinatorics of generalised row    cuts.   
 \begin{defn}
Given $r,c,m \in \mathbb{N}$,  we let $\Lambda_{n,d}(r,c,m)\subseteq \Lambda_{n,d}$ denote the set 
$$  
\{
  \lambda \in \Lambda_{n,d} \mid
   \lambda_j\leq c \leq \lambda_i, 
   \text{ for } 1\leq i \leq r {\text{ and }} r+1\leq j \leq d,   
    |\lambda^T| = m
\}
$$
and we let 
$$
\Lambda^+_{n,d}(r,c,m)=  \{
 \lambda \in \Lambda^+_{n,d} \mid
\lambda_r \geq c \geq \lambda_{r+1},     |\lambda^T| = m\}
$$
in other words, $\Lambda^+_{n,d}(r,c,m)= \Lambda^+_{n,d}\cap \Lambda_{n,d}(r,c,m)$.
Extending the above notation we denote by $\Lambda^+_{n,d}(0,c,0)$ the subset of $\Lambda^+_{n,d}$ consisting of all the partitions $\lambda$ such that $\lambda_1\leq c$.
\end{defn}

\begin{rmk}The subset  $ \Lambda^+_{n,d}(r,c,m) \subseteq  \Lambda^+_{n,d}$  can be thought of diagrammatically as in   \cref{asjksfjkdhsd}.  \end{rmk}
\begin{figure}[ht!]
$$\scalefont{0.8}
\begin{tikzpicture}[scale=1]
   \path (0,0) coordinate (origin);     
   \path (0,0)--++(-90:7*0.2)--++(180:2*0.2) coordinate (origin1);
    \draw[very thick] (origin1) --++(0:22*0.2);
  \path  (origin)--++(-90:7*0.2)--++(180:3*0.2)
  node {${r}$} ; 
    \path  (0,0)--++(0:9*0.2)--++(90:3*0.2) coordinate 
  node {${c}$} ; 
\path (0,0)--++(0:9*0.2)--++(90:2*0.2) coordinate (origin2);
    \draw[very thick] (origin2) --++(-90:15*0.2);
    \draw (origin) 
   --++(0:20*0.2)
      --++(-90:1*0.2)
 --++(180:1*0.2)
   --++(-90:1*0.2)
 --++(180:2*0.2) --++(-90:1*0.2)
 --++(180:2*0.2)
   --++(-90:2*0.2)
    --++(180:1*0.2) --++(-90:1*0.2)
 --++(180:4*0.2)
 --++(-90:1*0.2)
 --++(180:2*0.2)
    --++(180:1*0.2)
 --++(-90:3*0.2)
   --++(180:1*0.2)
 --++(-90:2*0.2) --++(180:1*0.2)
 --++(-90:2*0.2)
   --++(180:2*0.2)
    --++(-90:1*0.2) --++(180:3*0.2)
--++(90:15*0.2)
   ;
   \draw [decorate,decoration={brace,amplitude=3pt},xshift=-2pt]   (0,-1.3) --(0,-.0)  node [black,midway,xshift=-0.4cm]{$\lambda^T$} ; 
   \draw [decorate,decoration={brace,amplitude=3pt},xshift=-2pt]   (0,-3) --(0,-1.5)  node [black,midway,xshift=-0.4cm]{$\lambda^B$} ;  

   \end{tikzpicture}
$$
\caption{ 
A partition  $\lambda$ such that     $\lambda_{r} \geq c\geq\lambda_{r+1}$. 
}
\label{asjksfjkdhsd}
\end{figure}
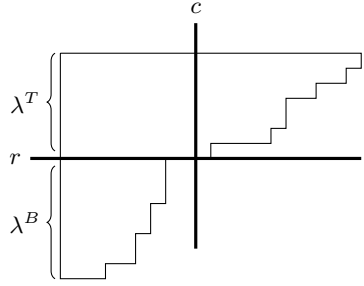

 \begin{prop}\label{sfadjfasdhjkdsfaj}
The map $\lambda\mapsto \lambda^T\times\lambda^B$ is a bijection between 
$\Lambda_{n,d}(r,c,m)$ and 
$\Lambda_{m,r}(r,c,m)\times\Lambda_{n-m,d-r}(0,c,0)$. Moreover, for $\lambda,\mu\in\Lambda_{n,d}(r,c,m)$, we have that $\lambda\trianglerighteq\mu$ if and only if $\lambda^T\trianglerighteq\mu^T$ and $\lambda^B\trianglerighteq\mu^B$.
\end{prop}
\begin{proof}
Clear from the definitions. 
\end{proof}

\begin{eg}For example, the map in \cref{sfadjfasdhjkdsfaj} takes the element in \cref{asjksfjkdhsd} to the pair of elements in \cref{asjksfjkdhsd2}.  
\end{eg}
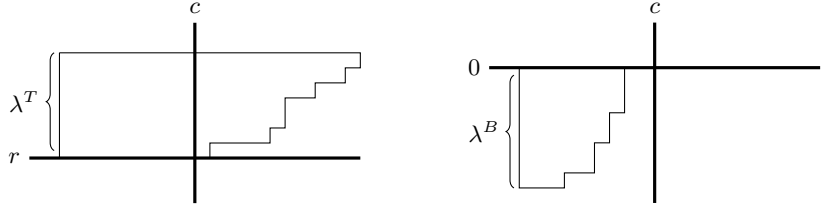
\begin{figure}[ht!]
$$\scalefont{0.8}
\begin{tikzpicture}[scale=1]
   \path (0,0) coordinate (origin);     
   \path (0,0)--++(-90:7*0.2)--++(180:2*0.2) coordinate (origin1);
    \draw[very thick] (origin1) --++(0:22*0.2);
  \path  (origin)--++(-90:7*0.2)--++(180:3*0.2)
  node {${r}$} ; 
    \path  (0,0)--++(0:9*0.2)--++(90:3*0.2) coordinate 
  node {${c}$} ; 
\path (0,0)--++(0:9*0.2)--++(90:2*0.2) coordinate (origin2);
    \draw[very thick] (origin2) --++(-90:12*0.2);
    \draw (origin) 
   --++(0:20*0.2)
      --++(-90:1*0.2)
 --++(180:1*0.2)
   --++(-90:1*0.2)
 --++(180:2*0.2) --++(-90:1*0.2)
 --++(180:2*0.2)
   --++(-90:2*0.2)
    --++(180:1*0.2) --++(-90:1*0.2)
 --++(180:4*0.2)
 --++(-90:1*0.2)
 --++(180:2*0.2)
    --++(180:8*0.2)  
 --++(90:7*0.2)
   ;
   \draw [decorate,decoration={brace,amplitude=3pt},xshift=-2pt]   (0,-1.3) --(0,-.0)  node [black,midway,xshift=-0.4cm]{$\lambda^T$} ; 
%
   \end{tikzpicture}
   \qquad \qquad
   \begin{tikzpicture}[scale=1]
   \path (0,0) coordinate (origin);     
   \path (0,0)--++(-90:7*0.2)--++(180:2*0.2) coordinate (origin1);
    \draw[very thick] (origin1) --++(0:22*0.2);
  \path  (origin)--++(-90:7*0.2)--++(180:3*0.2)
  node {${0}$} ; 
    \path  (0,0)--++(0:9*0.2)--++(-90:3*0.2) coordinate 
  node {${c}$} ; 
\path (0,0)--++(0:9*0.2)--++(-90:4*0.2) coordinate (origin2);
    \draw[very thick] (origin2) --++(-90:12*0.2);
    \path (origin)    --++(0:8*0.2)  --++(-90:7*0.2)
coordinate (origin6);
 \draw (origin6) 
    --++(180:1*0.2)
 --++(-90:3*0.2)
   --++(180:1*0.2)
 --++(-90:2*0.2) --++(180:1*0.2)
 --++(-90:2*0.2)
   --++(180:2*0.2)
    --++(-90:1*0.2) --++(180:3*0.2)
--++(90:8*0.2)
   ;
   
   \draw [decorate,decoration={brace,amplitude=3pt},xshift=-2pt]   (0,-3) --(0,-1.5)  node [black,midway,xshift=-0.4cm]{$\lambda^B$} ;  

   \end{tikzpicture}
$$
\caption{ The element of   $ \Lambda^+_{m,r}(r,c,m)
 \times
 \Lambda^+_{n-m,d-r}(0,c,0)$ obtained from the  element in \cref{asjksfjkdhsd} under the map in \cref{sfadjfasdhjkdsfaj}. }
\label{asjksfjkdhsd2}
\end{figure}
 
Let  $r,c,m \in \mathbb{N}$ be such that
$\Lambda^+_{n,d}(r,c,m)\neq \emptyset$.  
 We note that the set $\Lambda^+_{n,d}(r,c,m)$ has a unique maximal and a unique minimal element (under the dominance ordering on partitions).
 One can describe these partitions directly, however we use \cref{sfadjfasdhjkdsfaj} to make the statements simpler.  
 The unique maximal  and minimal   elements of 
 any non-empty  $\Lambda^+_{z,r}(0,c,0)$ are  equal to
$$
\alpha(r,c,z)=(  c^{\lfloor \tfrac{z}{c}\rfloor},  z -  c \lfloor \tfrac{z}{c}\rfloor ) 
\quad
\text{and}
\quad
\zeta(r,c,z)=
( r^{\lfloor \tfrac{z}{r}\rfloor}, z -  r \lfloor \tfrac{z}{r}\rfloor )' 
$$
respectively.  
For $r,c\geq z$ we have that $\alpha(r,c,z)=(z)$ and 
$\zeta(r,c,z)=(1^z)$.

\begin{prop}
If  $ \Lambda^+_{n,d}(r,c,m)\neq \emptyset,$ then it   
has a unique    maximal element 
 $$
\sigma:= \sigma(r,c,m)=( c^{r}, \alpha(d-r,c,n-m))  +(m-cr) 
$$
and a unique minimal element
$$
\gamma:= \gamma(r,c,m) =(  c^{r}, \zeta(d-r,c,n-m))  + \zeta(r,m-cr,m-cr).
$$  

\end{prop}

\begin{proof}
This follows from \cref{sfadjfasdhjkdsfaj}.  
\end{proof}

Having defined the maximal and minimal elements of 
$\Lambda^+_{n,d}(r,c,m)$ we now define
\begin{align*}
 \Sigma^+_{n,d}(r,c,m)=&\{
\mu \in \Lambda^+_{n,d} \mid \mu \trianglelefteq  \sigma  \}
\\
 \Gamma^+_{n,d}(r,c,m)
=&\{
\mu \in \Lambda^+_{n,d} \mid \mu \trianglerighteq \gamma   \}.
\end{align*}
The  set $\Sigma^+_{n,d}(r,c,m) 
$ (respectively 
$\Gamma^+_{n,d}(r,c,m)
 $) is clearly a saturated  
(respectively co-saturated) subset of $\Lambda^+_{n,d}$. 

We let  $\Gamma_{n,d}(r,c,m)$ (respectively $\Sigma_{n,d}(r,c,m)$)  
denote the sets of compositions 
which can be obtained from a partition in 
 $\Gamma^+_{n,d}(r,c,m)$ (respectively $\Sigma^+_{n,d}(r,c,m)$)   
 by permutation of  the rows $\{1,\ldots, r\}$ 
and the rows $\{r+1, \ldots , d\}$.  
 The  sets of  minimal and maximal elements of $\Lambda_{n,d}(r,c,m)$ are
those which are mapped to $\gamma$ and  $\sigma$ respectively under the map
 $\Lambda_{n,d}(r,c,m)\to  \Lambda_{n,d}^+(r,c,m)$.

\begin{eg}
The set $\Lambda_{10,4}^+(2,2,7)$ consists of two elements and is therefore equal to $\{\sigma,\gamma\}$ where
$\sigma=(5,2,2,1)$ and $ \gamma=(4,3,2,1)$.  
 \end{eg}

\begin{eg}
The set $\Lambda_{11,5}^+(3,2,9)$ consists of six elements 
$$(5,2^3)\quad
(4,3,2^2)\quad
(3^3,2)\quad
(5,2^2,1^2)\quad
(4,3,2,1^2)\quad
(3^3,1^2) 
$$
and here we have 
$\sigma=(5,2^3)$ and $ \gamma=(3^3,1^2) $.  
 \end{eg}

\begin{prop}
We have that 
$$ \Lambda^+_{n,d}(r,c,m) = \Sigma^+_{n,d}(r,c,m) \cap \Gamma^+_{n,d}(r,c,m)$$
\end{prop}
 \begin{proof}
It is clear that  $ \Lambda^+_{n,d}(r,c,m) \subseteq \Sigma^+_{n,d}(r,c,m) \cap \Gamma^+_{n,d}(r,c,m)$.   We now prove the reverse containment.  
 Suppose that $\mu \in \Lambda^+_{n,d}$ is such that
 $ 
   \gamma  
  \trianglelefteq \mu  \trianglelefteq
   \sigma  
. $ 
We have that
  $   \sum_{1\leq i \leq r}    \gamma  
 _i = m  =
 \sum_{1\leq i \leq r}    \sigma _i $ 
 and therefore 
  $$ m \leq  \textstyle \sum_{1\leq i \leq r}   \mu_i \leq m.$$ 
Therefore $ \sum_{1\leq i \leq r}   \mu_i=m$; putting this together with $\mu\trianglelefteq \sigma$ and  
  $    \sigma  
 _r \geq  c  $, we deduce that  $\mu_r \geq c$.  
  Similarly, we have that 
   $\mu\trianglelefteq \sigma$ and $    \sigma_{r+1} \leq c  $; therefore 
  $\mu_{r+1} \leq c$.  
  Therefore $\mu \in \Lambda^+_{n,d}(r,c,m)$, as required. 
 \end{proof}
\begin{defn}
 Given $\lambda,\mu \in \Lambda_{n,d}$ and 
 $1\leq r \leq d$, we say that $\lambda$ and $\mu$ admit a horizontal cut after the $r$th row if 
 $$
 \sum_{1\leq i \leq r} \lambda_i
 =
 \sum_{1\leq i \leq r} \mu_i. 
 $$
\end{defn}
  \begin{prop}\label{4.0}
  Let $\lambda,\mu \in \Lambda^+_{ n,d}$ be a pair of partitions that admits a  horizontal cut after the $r$th row.
If $\lambda \trianglerighteq \mu$, then 
 $\mu \in \Lambda_{n,d}^+(r,\lambda_r, |\lambda^T|)$.
 Moreover $\lambda \trianglerighteq \mu$ if and only if $\lambda^T \trianglerighteq \mu^T$ and $\lambda^B \trianglerighteq \mu^B$.
  \end{prop}

 \begin{proof}
Let  $\lambda, \mu \in \Lambda^+_{n,d}$ and suppose that  $\lambda$ and $\mu$ admit a horizontal cut after the $r$th row and $\la \trianglerighteq \mu$.  
 In which case, 
 $$  
 \mu_{r+1}\leq \lambda_{r+1}
 \leq \lambda_r \leq \mu_r 
 $$ 
and so   
$\mu \in \Lambda_{n,d}^+(r,\lambda_r, |\lambda^T| )$.  
The second statement is clear. 
\end{proof}

\subsection{Quotients of Schur algebras}\label{ACT2}
Given $r,c,m\in \mathbb N$, we
 have an idempotent decomposition of the identity as follows,
$$
1_{\Lambda_{n,d} \setminus  \Sigma_{n,d}({r,c,m})   } = 
 \sum_{ \mu \not \in  \Sigma_{n,d}({r,c,m})  }1_\mu
\qquad
1_{   \Sigma_{n,d}({r,c,m})    }
=
 \sum_{ \mu   \in  \Sigma_{n,d}({r,c,m})  }1_\mu.
$$
%
%
%
We shall consider the    quotient   algebras
$$S^\Bbbk ( \Sigma_{n,d}({r,c,m})    ) :=  S^{\Bbbk}_{n,d}/(S^{\Bbbk}_{n,d} 1_{\Lambda_{n,d} \setminus  \Sigma_{n,d}({r,c,m})  } S^{\Bbbk}_{n,d}),$$ 
We have a functor $f_{r,c,m}: 
S^{\Bbbk}_{n,d}\mhyphen {\rm mod} \to 
S^\Bbbk ( \Sigma_{n,d}({r,c,m})    ) \mhyphen {\rm mod}$ given by 
$$f_{r,c,m}(M)=  M/\langle1_{\Lambda_{n,d} \setminus  \Sigma_{n,d}({r,c,m})  } M\rangle .$$

\begin{prop}\label{4.1}
The algebra $S^\Bbbk( \Sigma_{n,d}({r,c,m}) ) $ is a quasi-hereditary 
algebra with identity $1_{\Sigma_{n,d}(r,c,m)}$.  The algebra is  free as a $\ZZ$-module with cellular basis
 $$
 \{ \xi_{\SSTS\SSTT}  \mid  \SSTS\in\SStd(\lambda,\mu), \SSTT\in \SStd(\lambda,\nu) \text{ for }\lambda \in   \Sigma^+_{n,d}({r,c,m})       , \mu,\nu \in \Sigma_{n,d}({r,c,m})   \}.
 $$
A full set of    non-isomorphic  simple, standard, and injective $S^\Bbbk( \Sigma_{n,d}({r,c,m})) $-modules are given by
 $$
f_{r,c,m}(      L(\lambda) )
\quad
f_{r,c,m}     ( \Delta(\lambda) )
\quad
 f_{r,c,m}  (   I(\lambda))
 $$
respectively, for $\lambda \in    \Sigma^+_{n,d}({r,c,m})   $.
 If $\lambda  \notin \Sigma^+_{n,d}({r,c,m})  $ we have that
$$ f_{r,c,m} (       L(\lambda))=0
\quad
f_{r,c,m} (    \Delta(\lambda)) =0
\quad
f_{r,c,m}    (   I(\lambda))=0.
 $$  
Moreover, we have that 
$$
[\Delta(\lambda):L(\mu)]_{S^{\Bbbk}_{n,d}} =[f_{r,c,m}(      \Delta(\lambda)):f_{r,c,m}      (L(\mu))
]_{S^\Bbbk( \Sigma_{n,d}({r,c,m}) )}.
$$ 
Given $M, N \in S^\Bbbk_{n,d}\mhyphen {\rm mod}$ belonging to $ \Sigma^+_{n,d}(r,c,m)$, we have that
 \[
{ \rm Ext}^j_{S^{\Bbbk}_{n,d}}(M,N)\cong
{ \rm Ext}^j_{S^\Bbbk( \Sigma_{n,d}({r,c,m}) )} (f_{r,c,m}(M),f_{r,c,m}    (N)).  
 \]
\end{prop}
\begin{proof}The set $  \Sigma_{n,d}({r,c,m}) $ is {\sf saturated} in the dominance ordering on partitions.
%
All the results now follow  from standard facts about   quotient functors \cite{Donkin}.
\end{proof}

\subsection{Subalgebras of Schur algebras}\label{ACT3}
Given $r,c,m\in \mathbb N$, we
define  the idempotent 
$$
 1_{   \Gamma^+_{n,d}({r,c,m})    }
=
 \sum_{ \mu   \in  \Gamma^+_{n,d}({r,c,m})  }1_\mu
$$
and associated idempotent subalgebras   as follows, 
$$S^\Bbbk_{n,d}(\Gamma^+_{n,d}({r,c,m}) ) := 1_{   \Gamma^+_{n,d}({r,c,m})    }  S^\Bbbk_{n,d} 1_{   \Gamma^+_{n,d}({r,c,m})    }.$$ 
We have a functor  
 $g_{r,c,m} :S^{\Bbbk}_{n,d}\mhyphen {\rm mod} \rightarrow S^\Bbbk_{n,d}(\Gamma^+_{n,d}({r,c,m}) )\mhyphen {\rm mod}$ given by  $$g_{r,c,m}(M)= 1_{ \Gamma^+_{n,d}({r,c,m}) }  M.$$  

\begin{prop}\label{4.2}
The algebra $S^\Bbbk (\Gamma^+_{n,d}({r,c,m}) )$ is a quasi-hereditary 
algebra with identity $1_{   \Gamma^+_{n,d}({r,c,m})    }$.
The algebra  is free as a $\ZZ$-module with cellular basis
 $$
 \{ \xi_{\SSTS\SSTT}  \mid  \SSTS\in\SStd(\lambda,\mu), \SSTT\in \SStd(\lambda,\nu) \text{ for }\lambda , \mu,\nu \in\Gamma^+_{n,d}({r,c,m})  \}.
 $$
A full set of    non-isomorphic  simple, standard, and injective $S^\Bbbk (\Gamma^+_{n,d}({r,c,m}) ) $-modules are given by
 $$
g_{r,c,m}(      L(\lambda) )
\quad
g_{r,c,m}      (\Delta(\lambda) )
\quad
 g_{r,c,m}     (I(\lambda))
 $$
respectively, for $\lambda \in    \Gamma^+_{n,d}({r,c,m})   $.
   We have that 
$$
[\Delta(\lambda):L(\mu)]_{S^{\Bbbk}_{n,d}} =[g_{r,c,m}(      \Delta(\lambda)):g_{r,c,m}      (L(\mu))
]_{S^\Bbbk( \Gamma^+_{n,d}({r,c,m}) )}.  
$$ 
Let $N, M \in S^\Bbbk_{n,d}\mhyphen {\rm mod}$  and suppose that $M\in \mathcal{F}_{\Gamma^+_{n,d}}(\Delta)$.  We have that
 \[
{ \rm Ext}^j_{S^{\Bbbk}_{n,d}}(M,N)\cong
{ \rm Ext}^j_{S^\Bbbk( \Gamma^+_{n,d}({r,c,m}) )} (g_{r,c,m}(M),g_{r,c,m}    (N)).  
 \]
 \end{prop}
 
 \begin{proof}
The set $ \Gamma^+_{n,d}({r,c,m}) $ is {\sf co-saturated} in the dominance ordering on partitions.  All the results now follow  from standard facts about the idempotent truncation functors \cite{Donkin}.
 \end{proof}

\subsection{Subquotient algebras of Schur algebras}\label{ACT4}
Given $r,c,m \in \mathbb{N}$, we define the idempotent
 $$1_{r,c,m}^{n,d}= \sum_{\mu \in \Lambda^+_{n,d}(r,c,m)}1_\mu,$$
and associated   subquotient algebra
$$S^\Bbbk ( \Lambda_{n,d}^+(r,c,m)) 
:= 1_{   \Gamma^+_{n,d}({r,c,m})    }  (S^\Bbbk_{n,d} / S^\Bbbk_{n,d}1_{  \Lambda_{n,d} \setminus \Sigma_{n,d}({r,c,m})  	 }S^\Bbbk_{n,d}	) 1_{   \Gamma^+_{n,d}({r,c,m})    }.$$ 
We have a functor 
 $h_{r,c,m} :S^{\Bbbk}_{n,d}\mhyphen {\rm mod} \rightarrow 
 S^\Bbbk_{n,d}( \Lambda_{n,d}^+(r,c,m) )\mhyphen {\rm mod}$ given by $$h_{r,c,m}(M)= 1_{\Gamma^+_{n,d}({r,c,m})}  (M /
 \langle 1_{  \Lambda_{n,d} \setminus \Sigma_{n,d}({r,c,m})  	 }M \rangle ) .$$  
\begin{prop}\label{4.2}
The algebra $S^\Bbbk_{n,d}(\Lambda_{n,d}^+(r,c,m) )$ is a
  quasi-hereditary algebra with identity $1_{r,c,m}^{n,d}$. The algebra is free as a $\ZZ$-module with cellular basis
 $$
 \{ \xi_{\SSTS\SSTT}  \mid  \SSTS\in\SStd(\lambda,\mu), \SSTT\in \SStd(\lambda,\nu) \text{ for }\lambda  , \mu,\nu \in\Lambda_{n,d}^+(r,c,m)   \}.
 $$
A full set of    non-isomorphic  simple, standard, and injective $S^\Bbbk_{n,d}(\Lambda^+_{n,d}({r,c,m}) ) $-modules are given by
 $$
h_{r,c,m} (     L(\lambda) )
\quad
h_{r,c,m}    (  \Delta(\lambda) )
\quad
 h_{r,c,m}    ( I(\lambda))
 $$
respectively, for $\lambda \in    \Lambda^+_{n,d}({r,c,m})  $.
   We have that 
$$
[\Delta(\lambda):L(\mu)]_{S^{\Bbbk}_{n,d}} =[h_{r,c,m}   (   \Delta(\lambda)):h_{r,c,m}(      L(\mu))
]_{S^\Bbbk ( \Lambda_{n,d}^+(r,c,m) )}.
$$ 
Let $ M, N  \in S^\Bbbk_{n,d}\mhyphen {\rm mod}$   belonging  to $\Sigma^+_{n,d}(r,c,m)$ such further suppose that $M\in \mathcal{F}_{\Gamma^+_{n,d}}(\Delta)$.  We have that
 \[
{ \rm Ext}^j_{S^{\Bbbk}_{n,d}}(M,N)\cong
{ \rm Ext}^j_{S^\Bbbk( \Lambda^+_{n,d}({r,c,m}) )} (h_{r,c,m}(M),h_{r,c,m}    (N)).  
 \]

\end{prop}
 
 \begin{proof}
We have that $\Lambda^+_{n,d}(r,c,m) = \Gamma^+_{n,d}(r,c,m)\cap \Sigma^+_{n,d}(r,c,m)$ and therefore the statement follows by composition of the arguments above. \end{proof}

\subsection{Isomorphisms between subquotients of Schur algebras  } \label{ACT5}

We now construct the isomorphism  between the subquotient algebras in which we are interested.  We first extend the combinatorics of cuts to semistandard tableaux.

\begin{defn} Given 
 $\lambda,\mu \in \Lambda^+_{n,d}(r,c,m)$, we have a bijective map 
   from  $   \SStd(\lambda,\mu)$ to $\SStd (\lambda^T, \mu^T) \times  \SStd  (\lambda^B, \mu^B)$  
 given by $\SSTS \mapsto  \SSTS^T \times \SSTS^B$, where
 \begin{itemize}
\item  $ \SSTS^T$ is obtained from  $\SSTS$ by deleting the $(r+1)$th, $(r+2)$th, \dots rows;
\item  $ \SSTS^B$ is obtained from  $\SSTS$ by deleting the first $r$ rows and 
replacing each entry $i$ with the entry   entry $i-r$.  
\end{itemize}
 \end{defn}

 \begin{eg}\label{semi}
The pair $( (7,5,3^2,2,1), (5,5,5,2,2,2))$ admit a horizontal cut after the 3rd row.    Given $\SSTS\in\SStd(\lambda,\mu)$  the leftmost tableau depicted in \cref{semifig}, we have that $ \SSTS^T  \in\SStd (\lambda^T, \mu^T)$ 
   and 
   $ \SSTS^B \in\SStd (\lambda^B, \mu^B)$  is as depicted in  \cref{semifig}.   
 \begin{figure}[ht!]
 $$
\Yvcentermath1 \young(1111123,22223,333,445,56,6)
 \quad \quad
\Yvcentermath1 \young(1111123,22223,333)
 \quad \quad
 \young(112,23,3)
 $$
 \caption{The tableau $\SSTS$ and the corresponding  tableaux  $ \SSTS^T $ and $ \SSTS^B$ for $r=3$, from \cref{semi}.}
\label{semifig}
 \end{figure}
 
 \end{eg}

 \begin{thm}\label{4.3}
The map  $$\varphi :   S^\Bbbk (\Lambda^+_{n,d}({r,c,m}) ) \to 
%
S^\Bbbk (\Lambda^+_{m,r}(r,c,m)
)
\times 
S^\Bbbk ( \Lambda^+_{n-m,d-r}(0,c,0))$$ given by 
\begin{align*} 
\varphi(\xi_{\SSTS\SSTT})= \xi_{ \SSTS^T  \SSTT^T } \times \xi_{ \SSTS^B  \SSTT^B }  
\end{align*}
is an isomorphism of $\Bbbk$-algebras.  
 \end{thm}
 \begin{proof}
 If $\SSTT \in \SStd(\lambda,\mu)$ and  $\SSTT(i,j)\neq0$, then this implies $i,j\in\{1,\ldots ,r\}$ or $i,j\in\{r+1,\ldots ,d\}$.  Therefore we can factorise the elements $\xi_{\lambda\SSTT} $ as follows, 
\begin{align*}
\xi_{\lambda\SSTT} &= \prod^{d}_ { i =1} \left(\prod_{j=1}^{d} e_{i,j}^{[\SSTT(i,j)]}\right)   = 
   \prod^{d}_ { i =r+1} \left(\prod_{j=r+1}^{d} e_{i,j}^{[\SSTT(i,j)]}\right)
\times    \prod^{r}_ { i =1} \left(\prod_{j=1}^{r} e_{i,j}^{[\SSTT(i,j)]}\right) 
\\ 
&=   \xi_{\lambda\SSTT^B} \times  \xi_{\lambda\SSTT^T}
\end{align*}
and similarly for the elements $\xi_{\SSTS\lambda}$.    
 Therefore, the map $\varphi$ 
    can  be seen to be given by taking the products of generators on the left-hand side to those of the right-hand side as follows:
 \begin{align*}
\varphi( 1_{\lambda }) &=
\begin{cases}
  1_{\lambda^T} \times 1_{\lambda^B} & \text{if }\lambda \in \Lambda_{n,d} (r,c,m)  \\
  0 & \text{otherwise} \\
  \end{cases}
\\
 \varphi(e^{[m]}_{i,i+1}) &=
\begin{cases}
 e^{[m]}_{i,i+1}\times  1^{n-m,d-r}_{0,c,0}	 			& \text{if }1\leq i \leq r-1 \\
 1^{m,r}_{r,c,m} \times e^{[m]}_{i-r,i-r+1} 			& \text{if }r+1\leq i \leq d \\ 
 \end{cases}
\\
\varphi(f^{[m]}_{i,i+1}) &=
\begin{cases}
 f^{[m]}_{i,i+1} \times 1^{n-m,d-r}_{0,c,0}			& \text{if }1\leq i \leq r-1 \\
 1^{m,r}_{r,c,m}\times f^{[m]}_{i-r,i-r+1} 			& \text{if }r+1\leq i \leq d \\ 
 \end{cases}.
\end{align*}
and the result follows. 
  \end{proof}

We immediately obtain a new result concerning the extension groups between Weyl and simple modules.  

 \begin{cor}\label{resul3}
 If 
 $\lambda,\mu$ admit a horizontal cut after the $r$th row, then
  we have that
 \[
{ \rm Ext}^k_{S^{\Bbbk}_{n,d}}(\Delta(\lambda),L(\mu)) \cong \bigoplus_{i+j=k}{ \rm Ext}^i_{S^{\Bbbk}_{m,r}}( \Delta( { \lambda^T}), L( { \mu^T}))
\otimes
{ \rm Ext}^j_{S^{\Bbbk}_{n-m,d-r}}( \Delta( { \lambda^B}), L( { \mu^B})).  
 \]
 \end{cor}   
\begin{proof}
This is immediate  from \cref{4.0,4.3}
\end{proof}

We also obtain the following result of Donkin \cite{Donkin1}. 
 
 \begin{cor}\label{resul1}
 If $\lambda,\mu$ admit a horizontal cut after the $r$th row, then
$$[\Delta(\lambda): L(\mu)]=
[\Delta(\lambda^T): L(\mu^T)]
\times
[\Delta(\lambda^B): L(\mu^B)].
$$ 
 \end{cor}
\begin{proof}
This follows from  \cref{4.0,4.3}
\end{proof}  

We also  recover the unquantised versions of   \cite{LM}  
and \cite[4.2(17)]{Donkin}.

 \begin{cor}\label{resul2}
 If 
 $\lambda,\mu$ admit a horizontal cut after the $r$th row, then
  we have that 
 \[
 { \rm Ext}^k_{S^{\Bbbk}_{n,d}}(\Delta(\lambda),\Delta(\mu)) \cong \bigoplus_{i+j=k}{ \rm Ext}^i_{S^{\Bbbk}_{m,r}}( \Delta( { \lambda^T}), \Delta( { \mu^T}))
\otimes
{ \rm Ext}^j_{S^{\Bbbk}_{n-m,d-r}}( \Delta( { \lambda^B}), \Delta( { \mu^B})).  
 \]
 \end{cor}   
\begin{proof}
This is immediate  from \cref{4.0,4.3}
\end{proof}  

\begin{rmk}[Removing a single row] 
We now consider the example of row cuts for $r=1$.  
In this case, the  isomorphisms above (and implications for decomposition numbers and extension groups) were proven in \cite{FangHenkeKoenig}. In this particularly simple case,  the  results can also be seen to  follow  by  tensoring with the determinant representation and applying a duality (as noted by Donkin in \cite[Appendix]{FangHenkeKoenig}). 
 \end{rmk}

%
%
%
%
%
%
\subsection{$p$-Kostka numbers}  \label{ACT6}
By \cref{4.1,4.2}, we know that injective, standard, and simple modules are all preserved under the functors  $h_{r,c,m}$ and the isomorphism $\varphi$.  It remains to check that the generalised symmetric powers are also preserved.
 
  \begin{thm}\label{resul3}
Given $\lambda,\mu \in \Lambda_{n,d}^+(r,c,m)$, 
we have that
$$ 
%
h_{r,c,m} ( {\Sym^\mu(\Bbbk^d)}) \cong
 h_{r,c,m}( \Sym^{ \mu^T}(\Bbbk^{r}))
  \otimes 
  h_{0,c,0}( \Sym^{  \mu^B}(\Bbbk^{d-r}))
  $$
and 
$$
h_{r,c,m} ( I(\lambda)  ) \cong
 h_{r,c,m}( I(\lambda^T ))
  \otimes 
  h_{0,c,0}(I(\lambda^B )
  $$
and so we conclude that the $p$-Kostka  numbers are preserved under generalised row cuts.  
 \end{thm}
 
 \begin{proof}
First, we note that   $K_{ \mu\lambda}\neq 0$ implies $\lambda \trianglerighteq \mu$.  
The isomorphism of injective modules is clear from \cref{4.1,4.2,4.3}. 
The result will therefore follow once we prove the isomorphism between the images of the generalised symmetric powers.
 Recall that the module $\Sym^\mu(\Bbbk^d)$ has basis 
$$ 
\{
\rho_{\SSTS  \SSTT} \mid
 \SSTS \in \SStd(\lambda,\nu) , 
\SSTT\in  \SStd(\lambda,\mu), {\lambda \in \Lambda^+_{n,d}}  ,{\nu \in \Lambda_{n,d}}   \}.  
 $$ 
Therefore $h_{  r,c,m}(\Sym^\mu(\Bbbk^d))$ is the module with basis 
$$ 
\{
\rho_{\SSTS   \SSTT} \mid
 \SSTS \in \SStd(\lambda,\nu), 
\SSTT \in \SStd(\lambda,\mu),   \lambda,  \nu \in \Lambda^+_{n,d} (r,c,m)
  \} 
$$
and, of course, one obtains similar  bases for both of  the modules 
$ h_{r,c,m}(\Sym^{ \mu^T}(\Bbbk^{r}))$
  and 
  $h_{0,c,0}( \Sym^{  \mu^B}(\Bbbk^{d-r}))$.

Any  
$\SSTT \in \SStd(\lambda,\mu)$ has the entry $s $ in each of   the first $c$ columns of the    $sth$ row  for each $1 \leq s \leq r$.  
 Therefore, any tableau $\stt$ such that $\mu(\stt)\neq 0$ must necessarily have entries $1,\ldots,m$ in  the first $r$ rows 
 and the entries 
  $m+1,\ldots,n$ in  the final $ d-r$ rows.  
 Therefore, for  $\lambda, \mu \in \Lambda_{n,d}^+(r,c,m)$,  we have that the set
$$
\{\sts \in \Std(\lambda) \mid \mu(\sts) \neq 0\}$$
is naturally in bijection with the set  
$$\{\stt\in \Std(  \lambda^T) \mid \mu^T(\stt) \neq 0\} \times 
\{\stu\in \Std(  \lambda^B) \mid \mu^B(\stu)\neq 0\}, 
$$
via the map $\varphi(s)=  \sts^T \times \sts^B$,   
where
 \begin{itemize}
\item  $ \sts^T$ is obtained from  $\sts$ by deleting the $(r+1)$th, $(r+2)$th, \dots rows;
\item  $ \sts^B$ is obtained from  $\sts$ by deleting the first $r$ rows and 
replacing each entry $i$ with the entry   entry $i-m$.  
\end{itemize}
Therefore, the map $:\SSTT \mapsto \SSTT^T \times \SSTT^B$ lifts to an isomorphism
$$\psi :
h_{r,c,m} ( {\Sym^\lambda(\Bbbk^d)}) \longrightarrow 
 h_{r,c,m}( \Sym^{ \lambda^T}(\Bbbk^{r}))
  \otimes 
  h_{0,c,0}( \Sym^{  \lambda^B}(\Bbbk^{d-r}))
  $$ given by  
$$\psi\left(\rho_{\SSTS\SSTT}\right)=
  \psi\left(\sum_{\stt \in \SSTT}\rho_{\SSTS \stt }\right)
=\left(\sum_{\stt^T \in \SSTT^T}
\rho_{ \SSTS^T \stt^T} \right)
\times
\left(\sum_{\stt^B \in \SSTT^B}
\rho_{ \SSTS^B \stt^B}
\right)
=
 \rho_{\SSTS^T\SSTT^T}
 \times
  \rho_{\SSTS^B\SSTT^B} 
.$$
 
To complete the proof, it is enough to observe that if $\lambda$ does not dominate  $\mu$ then either $\lambda^T$ does not dominate  $\mu^T$ or $\lambda^B$ does not dominate  $\mu^B$, by Proposition \ref{4.0}.
 Hence, by \cref{greenstuff},   we have that $K_{ \mu\lambda}=0=K_{ \mu^T\lambda^T}\cdot K_{ \mu^B\lambda^B}$.
  \end{proof}

\subsection{Generalised column cuts}

 \label{asdfjkjkfdfjhksdfdsjkhlafsdhj}
Given $\lambda\in \Lambda^+_{n,d}$  and $1\leq c \leq n$, we define partitions 
  $$\lambda^L= (\lambda'_1,\lambda_2',\ldots \lambda_c')' \quad \lambda^R= (\lambda'_{c+1},\ldots \lambda_{n}')'.$$  
 We say that a pair of partitions  $\lambda$ and $\mu$ 
admit a generalised column cut after the $c$th column if 
$$
 \sum_{1\leq i \leq c} \lambda_i'
 =
 \sum_{1\leq i \leq c} \mu_i' 
 $$
for some $1\leq c\leq n$.

  One can define similar  subsets of $\Lambda^+_{n,d}(r,c,m)$ and  generalise all the arguments and isomorphisms of the previous sections to cover these reduction theorems for generalised column cuts.   
%
%
However, it is also easy to deduce these    results  in two steps as follows.  
 If $c=1$, the isomorphisms are  easily deduced by tensoring with the determinant representation (plus the use of an idempotent truncation if $d<n$, see \cite{FangHenkeKoenig} for more details).  
  The result   now follows by applying this isomorphism along with the isomorphisms  of \cref{4.1,4.2,4.3}.
 These arguments are standard for such results, see \cite[Proof of Proposition 2.4]{FL}.
 We go through this argument more explicitly for $p$-Kostka numbers below.  

  
  \begin{cor}
  The $p$-Kostka  numbers are preserved under generalised column  cuts.  
  In other words,  $K_{\lambda\mu}=K_{\lambda^L\mu^L}K_{\lambda^R\mu^R}$.   
   \end{cor}

  \begin{proof}
Suppose that $ \lambda,\mu\in \Lambda_{n,d}^+$ 
are such that $\lambda \trianglerighteq \mu$ and  $( \lambda,\mu)$  admits a vertical cut after the $c$th column; we let $r=\lambda_c'$.
It is easy to see that $(\lambda, \mu)$ admits a horizontal cut after the $r$th row. 
From tensoring with the determinant representation (see also \cite[Corollary 9.1]{FangHenkeKoenig}), it follows that $p$-Kostka numbers are preserved under first column removal.  Therefore,
\begin{align*}K_{\lambda\mu}	&=K_{\lambda^T\mu^T}K_{\lambda^B\mu^B}\\
 				&=K_{(\lambda_1^T-c,\dots \lambda_r^T-c)
					(\mu_1^T-c,\dots \mu_r^T-c)
					}
					K_{\lambda^B\mu^B}
					\\
					&=K_{ \lambda^R \mu^R
					}
					K_{\lambda^B\mu^B}\\
					&=K_{ \lambda^R \mu^R
					}
					K_{(c^r,\lambda^B)(c^r,\mu^B)}\\
					&=K_{ \lambda^R \mu^R
					}
					K_{ \lambda^L  \mu^L }
 \end{align*}
where the first equality follows from \cref{resul3}; the second (respectively  fourth) equality follows from a total of $c$ applications of  first column removal \cite[Corollary 9.1]{FangHenkeKoenig} (respectively $r$ applications of first row addition \cref{resul3}); 
and the third and fifth equalities follows by definition and our choice of  $r=\lambda_c'$.     \end{proof}

  \begin{rmk}
  In the case $r=1$ or $c=1$, the above reduction theorems  for $p$-Kostka numbers were first proven in  \cite{FangHenkeKoenig}. 
   \end{rmk}

 \section{The Schur functor }\label{Sec:Sfunct} 
 When $d\geq n$, the symmetric group acts faithfully on   $1_\omega\Ten$
  and we obtain an isomorphic copy of $\Bbbk\mathfrak{S}_n $ as the idempotent subalgebra $  1_\omega S^\Bbbk_{n,d} 1_\omega $ of $S^\Bbbk_{n,d} $. 
  In this section, we recall how one can use this idempotent truncation 
  map (the  {\sf Schur functor}) to   the study of the
   representation theory of $\Bbbk\mathfrak{S}_n $.  
 
\subsection{The Murphy basis of the symmetric group}
  
Given $\stt \in \Std(\lambda)$, recall that  $d_\stt$ is the element of $\mathfrak{S}_n$ such that $(\stt^\lambda)d_\stt=\stt$. 
For $\lambda\in \Lambda^+_{n,d}$  we denote by  $x_\lambda$ the element of the group algebra of the symmetric group defined by $$x_\lambda=\sum_{x\in\mathfrak{S}_\lambda}x.$$  
 \begin{thm}[Murphy]\label{MUFDASSAFAS}
 The group algebra of the symmetric group  is free as a $\mathbb{Z}$-module with  
   basis  
 $$  
 \{ x_{\sts \stt} \mid x_{\sts \stt}:=d_{\sts} x_\lambda d^{-1}_{\stt},  \sts,\stt  \in \Std(\lambda) \text{ for } \lambda \in \Lambda^+_{n,d}\}		.	$$
    \end{thm}

 Recall our bijective map
   $\omega: \Std(\lambda) \to \SStd( \lambda, \omega)$.
 Suppose that   $\omega(\sts)= \SSTS$ and 
 $\omega(\stt)= \SSTT$. 
Under this identification we obtain an isomorphism  
$ 1_\omega S^{\Bbbk}_{n,d} 1_\omega \cong \Bbbk \mathfrak{S}_n $ 
given by 
 $  :\xi_{\SSTS \SSTT} \mapsto x_{\sts \stt}. 
 $
 Therefore the basis in \cref{MUFDASSAFAS} is a 
  cellular basis (in the sense of \cite{GL}) under the inherited cell structure (in other words, it satisfies the properties detailed in  \cref{sdajfsadjksfdahjksadfhjkafsdhjklsadfhlj}).
 In particular, we have the following.  
  \begin{defn} \label{definition: cell module2}
Given  $\lambda\in\Lambda_{n,d}^+$, we define the  {\sf Specht module} $S^ \lambda $  to be the left $\Bbbk\mathfrak{S}_n $--module  with  basis 
  $$\{x_{\sts\stt^\lambda} + \Bbbk\mathfrak{S}_n^{\vartriangleright \lambda} \mid \sts \in \Std(\lambda) \}$$
where $\Bbbk\mathfrak{S}_n^{\vartriangleright \lambda} $ is the $\Bbbk$-submodule with basis 
 $$\{x_{\stu\stv}   \mid \stu,\stv \in \Std(\mu), \mu \vartriangleright \lambda \}.$$
  Similarly, we define 
  the  {\sf dual Specht module} $S_ \lambda $  to be the left $\Bbbk\mathfrak{S}_n $--module  with  basis 
  $$\{\rho_{\SSTS\SSTT^\lambda} + 1_\omega\Ten^{\vartriangleright \lambda} \mid \SSTS \in \SStd(\lambda,\omega) \}$$
where $ 1_\omega\Ten^{\vartriangleright \lambda} $ is the $\Bbbk$-submodule with basis 
 $$\{\rho_{\SSTU\stv}   \mid \SSTU \in \SStd(\mu,\omega), \stv\in\Std(\mu), \mu \vartriangleright \lambda \}.$$
  \end{defn}

  \begin{defn} \label{definition: cell module3}
We say that $\lambda\in\Lambda_{n,d}^+$  is $p$-restricted if 
$\lambda_i- \lambda_{i+1} < p$
 for all $1\leq i <d$.  
 If $\lambda\in\Lambda_{n,d}^+$    is  not $p$-restricted, we say that it is $p$-singular.  
   \end{defn}
   
Each  Specht module $S^\lambda$  is equipped with the   bilinear form $\langle\ ,\ \rangle_\lambda$, inherited from the idempotent truncation.   
This form  is degenerate  if and only if   $\lambda$ 
is   $p$-singular.  
 Given a $p$-restricted 
  $\lambda\in\Lambda_{n,d}^+$, we define the  
  the  {\sf simple module} $D (\lambda) $ 
 to be  the quotient of the Specht  module $S^ \lambda$ by the radical of the   the   bilinear form $\langle\ ,\ \rangle_\lambda$.  
  By elementary properties of idempotent truncation functors, we have that
$$[S^\lambda  : D(\mu) ] =
[\Delta(\lambda) : L(\mu)]$$
for all $\lambda \in \Lambda^+_{n,d}$ and all $p$-restricted 
$\mu \in \Lambda^+_{n,d}$.  By \cref{resul1}, we have the following corollary
 (see also \cite{Donkin1}).  
  \begin{cor}
Let  $\lambda$ denote a partition of $n$ and let $\mu$
denote a $p$-regular partition of $n$.
If $(\lambda,\mu)$ admit a horizontal cut after the $r$th row, then
$$[S^\lambda  : D(\mu) ]
=[S^{\lambda^T}  : D(\mu^T ) ]    \times [S^{\lambda^B}  : D(\mu^B ) ]   .$$ 
 \end{cor}
\begin{proof}
This follows immediately from \cref{resul1} and the above.  
\end{proof}  
  
\subsection{Young permutation modules}
Given $\mu \in \Lambda_{n,d}$, we let $M(\mu)$ denote the image of the generalised symmetric power under the Schur functor, 
  $$M(\mu)= 1_\omega( \Sym^\mu(\Bbbk^d)).$$ 
We refer to these modules as the  {\sf Young permutation modules}.  
By definition, the module $M(\mu)$ has  basis   given by  the subset of   all 
 the vectors of weight $\omega$ in \cref{Crispysabsis}  
 as follows
 $$
 \{\rho_{\SSTS \SSTT} \mid \SSTS \in \SStd(\lambda,\omega) , \SSTT \in \SStd(\lambda,\mu ), {\lambda \in \Lambda^+_{n,d}} \}.  
$$
Under the identification of $\SStd(\lambda,\omega)$ and $\Std(\lambda) $, we recover Murphy's basis of these permutation modules \cite{Murphy}.   
 \begin{prop}[J. A. Green \cite{green}]
For $\lambda,\mu \in \Lambda^+_{n,d}$, the module $M(\mu)$ decomposes as a direct sum as follows
$$
M(\mu) = \bigoplus_{\lambda\vdash n}K_{\mu \lambda}Y(\lambda)
$$
where  $Y(\lambda)=1_\omega(I(\lambda))$; we refer to  the module $Y(\lambda)$
as the  {\sf indecomposable Young module of weight} $\lambda$.  
\end{prop}

  \begin{cor}
If $(\lambda,\mu)$ admit a horizontal cut after the $r$th row, then
$$[M(\lambda)  : Y(\mu) ]
=[M({\lambda^T})  : Y(\mu^T ) ]    \times [M(\lambda^B)  : Y(\mu^B ) ]   .$$ 
 \end{cor}
\begin{proof}
This follows immediately from \cref{resul3} and the above.  
\end{proof}

\subsection{The faithfulness of the Schur functor} 
The following theorem, proven in  \cite[Section 6.4]{KN01} and   \cite[Proposition 10.5]{Donkinext}, 
states the
  degree to which  cohomological information 
is preserved under the Schur functor.  
 
 \begin{thm}
Let $\Bbbk$ denote an algebraically closed field of characteristic $p\geq 3$.  
 The Schur algebra $S^\Bbbk_{n,d}$ is a $(p-3)$-faithful cover 
 (in the sense of \cite{Rouq}) 
 of the symmetric group, $\Bbbk\mathfrak{S}_n$.  That is,
 $$
{ \rm Ext}^i_{S^\Bbbk_{n,d}}
 (\Delta(\lambda), \Delta(\mu))
 \cong 
{ \rm Ext}^i_{\Bbbk\mathfrak{S}_n}
 (S^\lambda, S^\mu)
 $$
 for all $\lambda,\mu \in \Lambda^+_{n,d}$ and all $0\leq i \leq p-3$.  
 \end{thm}

   \begin{cor}
Let $\Bbbk$ denote an algebraically closed field of characteristic $p\geq 3$.  
 If $(\lambda,\mu)$ admit a horizontal cut after the $r$th row, then
 $$
{ \rm Ext}^i_{\Bbbk\mathfrak{S}_n}
 (S^\lambda, S^\mu)
 \cong 
\bigoplus_{i+j=k} 
{ \rm Ext}^i_{\Bbbk\mathfrak{S}_m}
 (S^{\lambda^T}, S^{\mu^T})\otimes 
 { \rm Ext}^j_{\Bbbk\mathfrak{S}_{n-m}}
 (S^{\lambda^B}, S^{\mu^B})
 $$
 for all $\lambda,\mu \in \Lambda^+_{n,d}$ and all $0\leq i \leq p-3$.  
 \end{cor}

  \begin{proof}
This follows immediately from \cref{resul2} and the above.  
\end{proof} 
  
  \begin{rmk}
  This result can be partially extended   to   $p= 2$,   \cite[Theorem 1.1]{LM}.
  \end{rmk}

  \begin{rmk}
These results can be extended to  cyclotomic Hecke algebras,  \cite{FS1}.
  \end{rmk}

\end{document}